\documentclass[twoside,11pt]{article}
\usepackage{color}
\usepackage{algorithm,bm}
\usepackage{enumerate}
\usepackage{graphicx}
\usepackage{epstopdf}
\usepackage{authblk}
\usepackage{amsfonts}
\usepackage[square,sort,comma,numbers]{natbib}

\usepackage{amsmath}    
\usepackage[table]{xcolor}
\usepackage[T1]{fontenc}
\usepackage[utf8]{inputenc}

\usepackage{epstopdf}
\usepackage{amsmath}
\usepackage{epsf}
\usepackage{graphicx}

\def\Tk{T}
\def\Tkp{T'}
\def\F{\mathcal F}
\def\s{V}

\def\ts{\bar{V}}
\def\vn{\Sigma_n} 
\def\th{g}
\def\tf{u}

\def\bsig{\overline{\Sigma}}
\newcommand{\aseq}{\mbox{$\;\stackrel{\mbox{\tiny\rm a.s.}}{=}\;$}}

\newcommand{\defeq}{\stackrel{def}{=}}

\usepackage{comment,url,algorithm,algorithmic,graphicx,subcaption,relsize}
\usepackage{amssymb,amsfonts,amsmath,amsthm,amscd,dsfont,mathrsfs,mathtools,nicefrac}
\usepackage{float,psfrag,epsfig,color,xcolor,url,hyperref}
\usepackage{epstopdf,bbm,mathtools,enumitem}

\newcommand{\dee}{\mathrm d}
\newcommand{\dt}{\,\dee t}

\newcommand{\dz}{\,\dee z}

\def\balign#1\ealign{\begin{align}#1\end{align}}
\def\baligns#1\ealigns{\begin{align*}#1\end{align*}}
\def\balignat#1\ealign{\begin{alignat}#1\end{alignat}}
\def\balignats#1\ealigns{\begin{alignat*}#1\end{alignat*}}
\def\bitemize#1\eitemize{\begin{itemize}#1\end{itemize}}
\def\benumerate#1\eenumerate{\begin{enumerate}#1\end{enumerate}}

\newenvironment{talign*}
 {\csname align*\endcsname}
 {\endalign}
\newenvironment{talign}
 {\csname align\endcsname}
 {\endalign}

\def\balignst#1\ealignst{\begin{talign*}#1\end{talign*}}
\def\balignt#1\ealignt{\begin{talign}#1\end{talign}}

\let\originalleft\left
\let\originalright\right
\renewcommand{\left}{\mathopen{}\mathclose\bgroup\originalleft}
\renewcommand{\right}{\aftergroup\egroup\originalright}

\def\tinycitep*#1{{\tiny\citep*{#1}}}
\def\tinycitealt*#1{{\tiny\citealt*{#1}}}
\def\tinycite*#1{{\tiny\cite*{#1}}}
\def\smallcitep*#1{{\scriptsize\citep*{#1}}}
\def\smallcitealt*#1{{\scriptsize\citealt*{#1}}}
\def\smallcite*#1{{\scriptsize\cite*{#1}}}

\def\mbb#1{\mathbb{#1}}

\def\mrm#1{\mathrm{#1}}

\def\reals{\mathbb{R}} 

\def\<{\left\langle} 
\def\>{\right\rangle}

\def\defeq{\triangleq} 

\def\nhalf{\nicefrac{1}{2}}

\def\norm#1{\left\|{#1}\right\|} 
\newcommand{\twonorm}[1]{\norm{#1}_2} 
\newcommand{\opnorm}[1]{\norm{#1}_{\text{op}}} 
\newcommand{\fronorm}[1]{\norm{#1}_{\text{F}}} 
\newcommand{\nucnorm}[1]{\norm{#1}_{*}} 
\newcommand{\inner}[1]{{\langle #1 \rangle}} 

\def\E{\mbb{E}} 

\DeclareMathOperator{\Tr}{Tr} 
\def\T{\top} 
\def\Var{\mrm{Var}} 

\newcommand{\Gsn}{\mathcal{N}}

\newcommand{\grad}{\nabla}
\newcommand{\Hess}{\nabla^2} 

\providecommand{\argmin}{\mathop\mathrm{arg min}}

\ifdefined\nonewproofenvironments\else
\ifdefined\ispres\else
\newtheorem{theorem}{Theorem}
\newtheorem{lemma}[theorem]{Lemma}
\newtheorem{corollary}[theorem]{Corollary}

\renewenvironment{proof}{\noindent\textbf{Proof.}\hspace*{.3em}}{\qed\\}
\newenvironment{proof-sketch}{\noindent\textbf{Proof Sketch}
  \hspace*{0.em}}{\qed\bigskip\\}
\newenvironment{proof-idea}{\noindent\textbf{Proof Idea}
  \hspace*{0.em}}{\qed\bigskip\\}
\newenvironment{proof-of-lemma}[1][{}]{\noindent\textbf{Proof of Lemma {#1}.}
  \hspace*{0.em}}{\qed\\}
\newenvironment{proof-of-corollary}[1][{}]{\noindent\textbf{Proof of Corollary {#1}.}
  \hspace*{0.em}}{\qed\\}
\newenvironment{proof-of-theorem}[1][{}]{\noindent\textbf{Proof of Theorem {#1}.}
  \hspace*{0.em}}{\qed\\}
\newenvironment{proof-attempt}{\noindent\textbf{Proof Attempt}
  \hspace*{0.em}}{\qed\bigskip\\}

\fi

\newtheorem{assumption}{Assumption}
\fi
\makeatletter
\@addtoreset{equation}{section}
\makeatother

\hypersetup{
  colorlinks,
  linkcolor={red!50!black},
  citecolor={blue!50!black},
  urlcolor={blue!80!black}
}

\newcommand{\eq}[1]{\begin{align}#1\end{align}}
\newcommand{\eqn}[1]{\begin{align*}#1\end{align*}}

\newcommand{\Ex}[1]{\mathbb{E}\left[#1\right]}

\newcommand{\ind}[1]{{\mathbbm{1}}_{\{ #1 \}} }
\newcommand{\abs}[1]{\left|#1\right|}

\newcommand{\handout}[5]{
  \noindent
  \begin{center}
    \framebox{
      \vbox{
        \hbox to 5.78in { {\bf \title } \hfill #2 }
        \vspace{4mm}
        \hbox to 5.78in { {\Large \hfill #5  \hfill} }
        \vspace{2mm}
        \hbox to 5.78in { {\em #3 \hfill #4} }
      }
    }
  \end{center}
  \vspace*{4mm}
}

\title{Normal Approximation for Stochastic Gradient Descent \\via Non-Asymptotic Rates of Martingale CLT\footnote{Authors are listed in alphabetical order. This paper provides a solution for an open problem formulated at the AIM workshop on Stein's method and Applications in High-dimensional Statistics, 2018 (problem (4) in page 2 of https://aimath.org/pastworkshops/steinhdrep.pdf by KB). We thank Jay Bartroff, Larry Goldstein, Stanislav Minsker, and Gesine Reinert for organizing a stimulating workshop, and the participants for several discussions.  MAE is partially funded by CIFAR AI Chairs program at the Vector Institute. }}
\author[1]{Andreas Anastasiou\thanks{a.anastasiou@lse.ac.uk}}
\author[2]{Krishnakumar Balasubramanian\thanks{kbala@ucdavis.edu}}
\author[3]{Murat A. Erdogdu\thanks{erdogdu@cs.toronto.edu}}
\affil[1]{Department of Statistics, London School of Economics and Political Science}
\affil[2]{Department of Statistics, University of California, Davis}
\affil[3]{Computer Science and Statistical Sciences, University of Toronto, and Vector Institute}

\begin{document}


\maketitle

\begin{abstract}
  We provide non-asymptotic convergence rates of the Polyak-Ruppert averaged stochastic gradient descent (SGD) to a normal random vector for a class of twice-differentiable test functions. A crucial intermediate step is proving a non-asymptotic martingale central limit theorem (CLT), i.e.,
  establishing the rates of convergence of a multivariate martingale difference sequence to a normal random vector, which might be of independent interest. We obtain the explicit rates for the multivariate martingale CLT using a combination of Stein's method and Lindeberg's argument, which is then used in conjunction with a non-asymptotic analysis of averaged SGD proposed in~\cite{polyak1992acceleration}.
  Our results have potentially interesting consequences for computing confidence intervals for parameter estimation with SGD
  and constructing hypothesis tests with SGD that are valid in a non-asymptotic sense.
\end{abstract}

\newpage
 \section{Introduction}\label{sec:intro}
 Consider the standard parametric population M-estimation or learning problem, of the form
\begin{align}\label{eq:main_prob}
\theta_* = \min_{\theta \in \mathbb{R}^d} \left\{ f(\theta)  = \E_Z [F(\theta, Z)]= \int F(\theta, Z) \, dP(Z) \right\}.
\end{align}
Here, the function $F(\theta, Z)$ is typically the loss function composed with functions from hypothesis class parametrized by $\theta \in \mathbb{R}^d$, and depends on the random variable $Z \in \mathbb{R}^d$. The distribution $P(Z)$ is typically unknown and hence the sample M-Estimator for the above problem involves sampling $N$ i.i.d. samples $\{Z_i\}_{i=1}^N$ from the distribution $P(Z)$ and computing the minimizer of the sample average~\citep{van2000asymptotic}. That is
\begin{align}\label{eq:samplemest}
\hat \theta_N = \argmin_{\theta \in \mathbb{R}^d} \left\{ \frac{1}{N} \sum_{i=1}^N F(\theta, Z_i)\right\}.
\end{align}
Such an approach is also called as Stochastic Average Approximation (SAA) method in the optimization literature~\citep{nemirovski2009robust}. When the loss function $F$ is the negative log-likelihood function, the above approach correspond to the popular Maximum Likelihood Estimator (MLE).  We emphasize that one still needs to specify an algorithm to compute the minimizer $\hat\theta_N$ in~\eqref{eq:samplemest}, in practice.  Asymptotic properties of the sample M-Estimator and the MLE, in particular consistency and asymptotic normality, have been studied for decades in the statistics literature, since the pioneering work of \cite{cramer1946mathematical}. The asymptotic normality result of the sample M-estimator in~\eqref{eq:samplemest}, in particular has proven to be extremely crucial for several inferential tasks. A critical drawback of such asymptotic results is that they are typically established for the actual minimizer $\hat \theta_N$ and \emph{not} the computational algorithm used.

Stochastic Gradient Descent (SGD) provides a direct method for solving the stochastic optimization problem \eqref{eq:main_prob} which also corresponds to the population M-estimation problem. Specifically, the SGD update equation is given by the following iterative rule: $\theta_{t} = \theta_{t-1} - \eta_t g(\theta_t)$, where $g(\theta_t)$ is the \emph{stochastic} gradient at  $\theta_t$ of the objective function $f(\theta)$.  SGD has been the algorithm of choice for parameter estimation in several statistical problems due to its simplicity, online nature and superior performance~\citep{nemirovski2009robust}. Indeed there has been an ever increasing interest in analyzing the theoretical properties of SGD in the learning theory literature in the last decade under various assumption on the problem structure. We refer the reader to~\cite{kushner2003stochastic, borkar2009stochastic, shapiro2009lectures, moulines2011non, duchi2011adaptive, rakhlin2011making, lan2012validation,ghadimi2012optimal, dieuleveut2016nonparametric} for a non-exhaustive list of some recent developments.  A notable drawback of such works is that they mainly focus on quantifying the accuracy of the SGD in terms of optimization and/or statistical estimation errors. 

While uncertainty quantification of the sample M-Estimators in~\eqref{eq:samplemest} has been a topic of intense studies in the statistical literature (with the literature being too large to summarize) there has been relatively less attention on establishing asymptotic normality of SGD, the practical algorithm of choice. Quantifying the uncertainty of SGD began in the works of~\cite{chung1954stochastic, sacks1958asymptotic, fabian1968asymptotic, ruppert1988efficient, shapiro1989asymptotic}, with~\cite{polyak1992acceleration} providing a definitive result. A crucial step in~\cite{polyak1992acceleration}, for uncertainty quantification is that of establishing the asymptotic normality of the averaged SGD iterates, which in turn depends on a Martingale Central Limit Theorem established in the probability theory literature~\citep{liptser2012theory}. Recent works, for example,~\cite{chen2016statistical, su2018statistical, duchi2016local, toulis2017asymptotic,fang2018online}, leverage the asymptotic normality analysis of~\cite{polyak1992acceleration} and provide confidence intervals for SGD that are valid in an asymptotic sense.
 
While asymptotic inference based on the  asymptotic normality result is interesting and leads to qualitative confidence intervals, their validity is often times questionable due to their asymptotic nature; in practical scenarios the SGD algorithm is only run for a \emph{finite} number of steps. Indeed, even for the well-studied case of the sample M-estimator (or MLE) in~~\eqref{eq:samplemest}, the role of asymptotic normality, i.e., the case of $N \to \infty$, for inference is questionable because of its \emph{qualitative} nature; see for example \cite{geyer2013asymptotics, kuchibhotla2018deterministic}. A more \emph{quantitative} approach is to obtain explicit bounds for the rate of convergence to normality in a suitable metric. Indeed such results for the MLE have recently been obtained recently. Using Stein's method, ~\cite{anastasiou2018bounds} establish the rate of convergence to normality of $\hat\theta_N$ in the case of multivariate MLE. 
A main drawback of such a result, similar to the asymptotic results, is that the rates are established only for the actual minimizer $\hat \theta_N$, and not the estimator that is computed  in practice using a specific algorithm.  



In this work, we study the problem of \emph{quantifying} the rate of convergence to normality of the Polyak-Ruppert averaged SGD algorithm. Indeed SGD and its several variants are the practical algorithm of choice used to solve the stochastic optimization and population M-estimation problems in practice, and it is crucial to understand their finite sample convergence to normality for the purpose of non-asymptotic inference.  A main step in order to obtain such results is to first establish the rate of convergence to normality of a multivariate martingale difference sequence. We use a combination of Stein's method and Lindeberg's telescoping sum argument to establish such a result. We then adapt the proof of~\cite{polyak1992acceleration} and use it in conjunction with our martingale result to provide \emph{qualitative} bounds for the SGD iterates to normality. Our results have consequences for constructing confidence intervals for parameter estimation via SGD, that are valid in a non-asymptotic sense. Furthermore, a wide variety of statistical hypothesis tests could also be formulated based on functionals of solutions to convex optimization problems~\citep{goldenshluger2015hypothesis} that are typically solved using SGD in practice; we refer the reader to the excellent survey on this topic by~\cite{juditsky2018lectures}. Our results could be used in conjunction with such tests to obtain practical quantitative hypothesis tests that are valid in a non-asymptotic sense.

\vspace{0.1in}
\noindent \textbf{Our Contributions:} To summarize the discussion above, in this paper, we make the following contributions.
\begin{itemize}
\vspace{-0.1in}
\item In Theorem~\ref{thm:martingale-clt}, Corollaries~\ref{cor:clt-psd-cov} and \ref{cor:clt-non-psd-cov},
  we prove a non-asymptotic multivariate martingale CLT, i.e.,
  we establish the explicit rates of convergence of a multivariate martingale difference sequence to
  a normal random vector for the class of twice differentiable functions.

\item In Theorems~\ref{thm:linearsetting} and~\ref{thm:sgdsetting},
  we prove the rate of convergence of the Polyak-Ruppert averaged SGD iterates
  to a normal random vector for solving system of linear equations and optimizing strongly-convex functions respectively.
\end{itemize}

The rest of the paper is organized as follows. Section~\ref{sec:notation} introduces our notation. Section~\ref{sec:martingale-clt} contains all the relevant results for multivariate martingales and Section~\ref{sec:sgd-rates} contains our results on SGD. We conclude the paper in Section~\ref{sec:conclusion} with a brief discussion and future work.
All the proofs are provided in the Appendix Sections~\ref{sec:proof-clt} - \ref{appendixend}.



\subsection{Notation}\label{sec:notation}
For a real vector $v \in \reals^d$ and a real tensor
$T \in \reals^{d_1 \times d_2 \times \cdots \times d_m}$,
we define the operator norm as
$\opnorm{v} \defeq \twonorm{v}$,
and $\opnorm{T} = \sup_{\twonorm{u}=1}\opnorm{T[u]}$ defined recursively.
For matrices $A,B \in \reals^{d \times d}$, we use $A\succ B$ ($A\succeq B$) to indicate that $A-B$ is positive (semi)definite.
We use $\twonorm{A}$, $\fronorm{A}$, $\nucnorm{A}$ to denote the operator, Frobenius, and nuclear norms, respectively.
For a $k$ times differentiable function $f$ and $i\geq 1$, we define
\eq{
  M_0(f) = \sup_{x\in\reals^d}\opnorm{f(x)},\ 
  \text{ and }\ 
  M_i(f) = \sup_{x,y\in\reals^d,x\neq y}\tfrac{\opnorm{\nabla^{i-1}f(x)
      - \nabla^{i-1} f(y)}}{\twonorm{x-y}}.
}
Additionally, throughout the text we let $[n]$ denote the set $[n] = \{1,2,...,n\}$.
For random variables $X$ and $Y$, we use $X\sim\Gsn_d(\mu, \Sigma)$ to indicate that
$X$ is a $d$-variate Gaussian random vector with mean $\mu \in \reals^d$,
and covariance $\Sigma \in \reals^{d \times d}$, and we use $X\aseq Y$ to denote
$X$ is equal to $Y$ almost surely.

\section{Convergence Rates of a Multivariate Martingale CLT}
\label{sec:martingale-clt}
In this section, we prove a multivariate martingale central limit theorem (CLT)
with explicit rates and constants.
Convergence rates of univariate martingale CLT have been studied extensively,
and even a synthetic review would go beyond the page limits --
see for example \cite{bolthausen1982exact,rinott1999some,chow2012probability,mourrat2013rate,hall2014martingale,rollin2018quantitative} and the references therein.
In the standard literature, convergence rates are established
for one-dimensional random variables using Lindeberg's telescoping sum argument \citep{bolthausen1982exact,mourrat2013rate,fan2019exact}.
In order to obtain convergence rates with explicit constants, we adapt an approach from \cite{rollin2018quantitative}
to the multivariate setting, which is
based on a combination of Stein's method \citep{stein1986approximate}
and Lindeberg's telescoping sum argument \citep{bolthausen1982exact}.
We state the following non-asymptotic multivariate martingale CLT.
\begin{theorem}\label{thm:martingale-clt}
  Let $X_1, X_2, ..., X_n\in \reals^d$ be a martingale difference sequence
  adapted to a filtration $\F_0,\F_1,...,\F_n$ with $\s_k = \E[X_kX_k^\T|\F_{k-1}]$ almost surely.
  Denote their summation by $S_n = \sum_{i=1}^nX_i$, and for $k \in [n]$,
  the partial covariance by $P_{k} = \sum_{i=k}^n\s_i$,
  and the variance of the summation by $\vn = \Var(S_n)$.
  If we assume that
  \eq{\label{asmp:almost-sure} 
    P_1 = \vn \ \text{ almost surely},
  }
    then for $Z \sim \Gsn_d(0, I)$
   and $h : \reals^d \to \reals$ a twice differentiable function,
  we have
  \eq{\label{eq:mart-clt-bound}
    \abs{\Ex{h(\vn^{-\nhalf}S_n) - \Ex{h(Z)}}}
    \leq \frac{3\pi}{8}\sqrt{d}
    M_2(h)\sum_{k=1}^n \Ex{\|\vn^{\nhalf} P_k^{-1} \vn^{\nhalf}\|_2^{\nhalf}
      \|\vn^{-\nhalf}X_k\|_2^3}.
  }
\end{theorem}
Theorem~\ref{thm:martingale-clt} provides a non-asymptotic martingale CLT result under
the assumption~\eqref{asmp:almost-sure}. We emphasize that all the constants in the bound are explicit; yet the bound cannot be expressed in terms of the Wasserstein distance
since the second derivative of the test function $h$ appears on the right hand side.
This is the main difference between our result and that of \cite{rollin2018quantitative};
the bounds in \cite{rollin2018quantitative} are for the univariate case and in Wasserstein metric whereas Theorem~\ref{thm:martingale-clt} establishes the bounds in the multivariate setting but the resulting bound cannot be expressed in terms of Wasserstein distance
due to a subtlety of the Stein's method in high dimensions.

The function $h$ in Theorem~\ref{thm:martingale-clt} is often referred
to as the \emph{test function}, and the bound \eqref{eq:mart-clt-bound}
depends only on its smoothness, i.e., we require $M_2(h) < \infty$.
The remaining terms will be determined by the characteristics of the
martingale difference sequence.
For example, for a sequence with positive definite conditional covariances,
Theorem~\ref{thm:martingale-clt} yields the following result.

\begin{corollary}\label{cor:clt-psd-cov}
Instantiate the notation and assumptions of Theorem~\ref{thm:martingale-clt}.
  For a martingale difference sequence satisfying $\alpha I \preceq \s_k  \preceq \beta I$
  almost surely for all $k \in [n]$ and
  $\E\big[\twonorm{X_k}^3\big]\leq \gamma d^{3/2}$,
  we have
  \eq{\label{eq:mart-clt-psd-bound}
    \abs{\Ex{h(\vn^{-\nhalf}S_n) - \Ex{h(Z)}}}
    \leq& \frac{3\pi\gamma\sqrt{\beta}}{4\alpha^{2}}
    M_2(h)\frac{d^2}{ \sqrt{n}}.
  }
\end{corollary}
The above bounds yields a convergence rate of order
${d^2}/{\sqrt{n}}$ for any smooth test function, where $d$ is the dimension of the parameters and $n$ is the number of samples in the martingale sequence.
We are not aware of a multivariate martingale CLT result to compare the dimension dependence, but
our result improves upon the dependence in \cite{reinert2009multivariate} where
it is of order $d^3$ for the standard multivariate CLT for the independent random vectors.
We emphasize that the dimension dependence of order $d^2$ is a result of
the particular Stein equation being used in the proof.

In Section~\ref{sec:sgd-rates}, we will use Corollary~\ref{cor:clt-psd-cov}
to prove a non-asymptotic normality result for SGD.
When translated to the optimization terminology,
$n$ will denote the number of iterations in a stochastic algorithm,
and $d$ will denote the dimension of the parameters.

\subsection{Relaxing the Assumptions}
\label{sec:relax-cond-variance}
Theorem~\ref{thm:martingale-clt} and the consequent corollary
rely on the assumptions that (i) the eigenvalues of the conditional covariances $\s_k$ are bounded away from 0,
i.e., $\s_k \succeq \alpha I$,
and (ii) the summation of the conditional covariances are deterministic, i.e., $P_1 \aseq \vn$.
Even though these assumptions are satisfied
in the setting of \cite{polyak1992acceleration} for stochastic gradient algorithms,
we discuss relaxations to the assumptions which may be used to extend our results to different settings,
or may be of independent interest.
The following corollary provides a relaxation to the assumption $\s_k \succeq \alpha I$,
at the expense of introducing a stronger upper bound on the third conditional moment.
\begin{corollary}\label{cor:clt-non-psd-cov}
  Instantiate the notation and assumptions of Theorem \ref{thm:martingale-clt}.
  If we further assume that there are constants $\beta$ and $\delta$ such that
  \eq{\label{asmp:stronger-3rd-moment}
    \Ex{\twonorm{X_k}^3 | \F_{k-1}} \leq \beta \vee \delta \Tr(\s_k) \ \text{ almost surely,}
  }
  then, we have
  \eq{\label{eq:mart-clt-non-psd-bound}
    &\abs{\Ex{h(\vn^{-\nhalf} S_n) - \Ex{h(Z)}}}\\
    \nonumber
    &\ \leq
    2\frac{M_1(h)}{\sqrt{n}}\!\Tr(\tfrac{1}{n}\vn)^{\nhalf}
    + \frac{3\pi }{4}\delta\sqrt{d}n
    M_2(h) \|\vn^{-\nhalf}\|^3_2
    \big[ \Tr(\tfrac{1}{n}\vn) + \beta^{2/3} \big ]
    .
  }
\end{corollary}

Compared to Corollary~\ref{cor:clt-psd-cov},
the above result assumes that the test function is both Lipschitz and smooth.
One should think of $\vn$ as of order $n$ since it is the variance of the summation $S_n$; thus,
the resulting bound \eqref{eq:mart-clt-non-psd-bound} still decays with rate $\sqrt{n}$.
The parameters in the assumption \eqref{asmp:stronger-3rd-moment} will depend on the dimension $d$;
in a typical application, $\beta$ and $\delta$
will be of order $d^{3/2}$ and $\sqrt{d}$, respectively.
Therefore, the first term on the right hand side of \eqref{eq:mart-clt-non-psd-bound}
is of order $\sqrt{d/n}$ and the second term is of order $d^2/\sqrt{n}$; therefore,
the second term will dominate and determine the convergence rate.

The assumption $P_1 \aseq \Sigma$ as given in \eqref{asmp:almost-sure} holds
for the stochastic gradient schemes that we consider in this paper,
but it may not hold for a class of specialized algorithms where
the randomness introduced through the current iteration depends on the previous step; thus,
the conditional covariances are random variables and
their summation is not deterministic.
This may be the case, for example,
when the batch size in SGD is increased over each iteration by
including previously unused samples in the updates \citep{erdogdu2015convergence}.

Relaxing the assumption \eqref{asmp:almost-sure} has been the focus of many papers;
see for example \cite{bolthausen1982exact,mourrat2013rate,rollin2018quantitative}.
Since this assumption is conveniently satisfied for the stochastic algorithms that we consider in the current paper,
we omit a rigorous relaxation of this condition.
However, a classical approach for this task
would rely on a construction first introduced by \cite{bolthausen1982exact},
and also used in \cite{mourrat2013rate,rollin2018quantitative}.
The resulting bound would be identical to the right hand size of \eqref{eq:mart-clt-non-psd-bound}
with an additional error term $\E[\nucnorm{I - \vn^{-1}P_1}]^{\nhalf}$
quantifying the difference between $P_1$ and $\vn$.

\section{Rates to Normality for Stochastic Gradient Descent}
\label{sec:sgd-rates}
As discussed in Section~\ref{sec:intro}, the rates of convergence of the Martingale CLT established in Theorem~\ref{thm:martingale-clt} plays a crucial role in obtaining non-asymptotic results on the convergence to normality of the SGD iterates. We elaborate about this in this section. Specifically, we consider the problem of minimizing a smooth and strongly-convex function $f: \mathbb{R}^d \to \mathbb{R}$ by stochastic gradient descent. Specifically, we make the following assumptions on the function $f(\theta)$.
\begin{assumption}  \label{assume:forsgd}
We assume that the function $f(\theta)$ satisfies, for some constants $L, L_H,\mu > 0$ and any points $\theta, \theta' \in \mathbb{R}^d$:
  \begin{itemize}
 \item Strong convexity: $f(\theta) -f(\theta') -\<\nabla f(\theta'), \theta-\theta'\> \geq \frac{\mu}{2}\|\theta - \theta' \|_2^2$.
 \item Hessian-Smoothness: $\|\nabla^2 f(\theta')-\nabla^2 f(\theta)\|_2 \le L_H \|\theta'-\theta\|_2, \quad \forall~\theta,\theta' \in \mathbb{R}^d.
$
\end{itemize}
\end{assumption}
Recall that we denote the minimizer as $\theta^* = \underset{\theta \in \mathbb{R}^d}{\argmin}~~f(\theta)$ and we consider the following SGD updates along with the Polyak-Ruppert averaging scheme. Starting with $\theta_0 \in \mathbb{R}^d$ we define the following sequence
\begin{align}\label{eq:sgdupdate}
\theta_t = \theta_{t-1} - \eta_t g(\theta_{t}), \qquad \bar{\theta}_t = \frac{1}{t} \sum_{i=0}^{t-1} \theta_i,
\end{align}
where, with $\zeta_{t}$ being a sequence of mean-zero i.i.d random vectors, we have
\begin{align}\label{eq:noisygrad}
g(\theta_{t}) = \nabla f(\theta_{t-1}) + \zeta_{t}.
\end{align}
This covers the popular framework of sub-sampling based stochastic gradient descent and zeroth-order stochastic gradient descent as well. Before proceeding to analyze the general SGD, it is instructive to consider the following linear setting, following the strategy in~\cite{polyak1992acceleration}. As will be seen in Section~\ref{sec:sgdsetting}, the proof of the general setting will largely follow from the proof of the linear setting considered in Section~\ref{sec:linearsetting}. 
\subsection{Linear Problem Setting}\label{sec:linearsetting}
To gain intuition, we first consider the case of solving the system of linear equations of the form $A \theta = b$ where the matrix $A \in \mathbb{R}^{d \times d}$ is assumed to be positive definite. Let $\theta^*$ denote the true solution of this linear system. Following~\cite{polyak1992acceleration}, we consider the following stochastic iterative algorithm for obtaining a solution for the linear system:
\begin{align}
 \theta_t & = \theta_{t-1} - \eta_t y_t, \qquad y_t = A \theta_{t-1}-b+\zeta_t, \label{eq:linaverage} \\
 \bar{\theta}_t &= \frac{1}{t} \sum_{i=0}^{t-1} \theta_i. \nonumber 
\end{align}
Here, with a slight abuse of notation, $\zeta_t$ is a random perturbation to the residual term $A \theta_{t-1}-b$. Furthermore, we denote by the tuple $(\Omega, \mathcal{F}, \mathcal{F}_t, \mathbb{P})$ an increasing sequence of Borel fields and assume that $\zeta_t$ is a martingale difference sequence adapted to the filtration $\mathcal{F}_t$. 
We now proceed to provide the main result of this section. Let $\Delta_t = \theta_t - \theta^*$ and note that we have $\Delta_t = \Delta_{t-1} - \eta_t \left[A \Delta_{t-1} + \zeta_t\right]$. We also denote the average residual by
\begin{align*}
\bar{\Delta}_t &= \frac{1}{t} \sum_{i=0}^{t-1} \Delta_i.
\end{align*}
We now present our non-asymptotic result on the rate of convergence of the vector $\sqrt{t}~\bar{\Delta}_t $ to a normal random vector $Z$. Define,
\begin{align}\label{eq:temp}
  \varrho(\eta,t) \defeq \sum_{j=1}^{t-1} \Bigg\{  e^{ -2c_1 \sum_{i=j}^{t-1} \eta_i     }
  + \Bigg[\frac{C'}{\eta_j^{1-c_2}} \sum_{i=j}^t m_j^t e^{-\lambda m_j^t} (m_j^i -m_j^{i-1})\Bigg]^2 \Bigg\}   
\end{align}
where $m_j^i \defeq \sum_{k=j}^i \eta_k$ and $\eta_j$ denotes the step-size (or learning rate) from~\eqref{eq:linaverage}. Then, we have the following result.

\begin{theorem}\label{thm:linearsetting}
Consider solving the system of linear equations $A\theta = b$ using the iterative updates in~\eqref{eq:linaverage}. Let the matrix $A$ be positive definite, with $0 < c = \lambda_{\min}(A)  \leq \lambda_{\max}(A) =C < \infty$. Furthermore, let the noise $\zeta_t$ satisfy the following assumptions:
\begin{align*}
\E\left[\zeta_t | \mathcal{F}_{t-1}\right] = 0,  \quad \E\left[\| \zeta_t\|_2^2| \mathcal{F}_{t-1}\right] \leq K_d <\infty,
\quad \E\left[\zeta_t \zeta_t^\top|\mathcal{F}_{t-1}\right] \aseq V. 
\end{align*}
Then, for some universal constants $K, K_2, C',c_1 >0$, $c_2 \in (0,1)$, 
a standard normal random vector $Z$, 
and for a twice differentiable function $h : \mathbb{R}^d \to \mathbb{R}$,
we have the following non-asymptotic bound, quantifying the rate of convergence of the iterates,
\begin{align*}
  & \E\left[|h(\sqrt{t}~\bar{\Delta}_t) - h(A^{-1}V^{\nhalf}Z)|\right]
    \leq  \sum_{k=1}^t \E \left[ \frac{ 1.18 \sqrt{d} M_2(h)~
    \big\|  \left[A^{-1} V A^{-1}\right]^{-\nhalf}\zeta_{k}\big\|_2^3}{t~\sqrt{(t-(k-1))}} \right] \\
& \; + \frac{M_1(h)}{\sqrt{t}}\left[\frac{K_2\left[ \E\|\Delta_0\|_2\right]}{\eta_0} + \sqrt{ K_d~K \varrho(\eta,t)  }\right]+ \frac{M_2(h)}{t}\left[\frac{K_2^2\left[ \E\|\Delta_0\|_2^2\right]}{\eta_0^2} + K_d~K \varrho(\eta,t)\right].
\end{align*}

\end{theorem}
The above bound quantifies the non-asymptotic rate of convergence of $\sqrt{t}~\bar{\Delta}_t $ to normality, for any given matrix $A, V$. The rate is presented in a way so that it explicitly quantifies the dependence on the learning rate $\eta_j$ and it holds for any $t \geq 1$. The current bound depends on the dimension $d$ through the following terms: the factor $\sqrt{d}$ in the first term; the spectrum of $A^{-1}VA^{-1}$; and expected norms $\E[\| \zeta_k\|_2]$ (through $K_d$) and $\E\left[\| \zeta_k\|_2^3\right]$. Under conditions similar to Corollary~\ref{cor:clt-psd-cov}, we now provide another corollary in order to gain some intuition on the dependence of the bound on the dimensionality $d$. The proof of the corollary below follows from Corollary~\ref{cor:clt-psd-cov} and the proof of Theorem 1 in~\cite{polyak1992acceleration}.

\begin{corollary}\label{corr:lineareg}
  Instantiate the notation and assumptions of Theorem \ref{thm:linearsetting}.
Assume that $A$ and $V$ are such that $\alpha I \preceq \left[A^{-1}V A^{-1} \right] \preceq \beta I$ and $\E\big[\twonorm{\zeta_k}^3\big]\leq \gamma d^{3/2}$. Furthermore, assume that $\eta_t = \eta t^{-c_3}$ for some $c_3 \in (0,1)$. Then, we have $K_d \leq d \beta$ and 
\begin{align*}
\E\left[|h(\sqrt{t}~\bar{\Delta}_t) - h(A^{-1}V^{\nhalf} Z)|\right] \leq \frac{2.36~\gamma\sqrt{\beta}}{\alpha^{2}}
    M_2(h)\frac{d^2}{ \sqrt{t}} + K_4~M_1(h)~\sqrt{\frac{d}{t}} + K_5~M_2(h)~\frac{d}{t},
\end{align*}
where $K_4$ and $K_5$ are constants (that only depend on $c_2, c_3, \eta, \beta, K,C', K_2$ and $\E\left[ \| \Delta_0\|_2^2\right]$) that are independent of $d$ and $t$.
\end{corollary}
We emphasize that our goal through Corollary~\ref{corr:lineareg} is to demonstrate mainly the dependence of the rate of averaged SGD to normality, on $d$ and $t$. Hence, careful attention is not paid to get the exact constants $K_4$ and $K_5$. Note that, under the setting assumed in Corollary~\ref{corr:lineareg}, the dominant term is the first term which is of the order $ \mathcal{O}(d^2/\sqrt{t})$. It would be interesting to explore different assumptions on the spectrum of $A$ and $V$ which might help reduce the dimension dependence. Finally,  the assumption on the third-moment of $\zeta_k$ is an additional assumption which is not required in the asymptotic results proved in~\cite{polyak1992acceleration}. Indeed this is reminiscent of the Berry-Essen type bounds required for rate of CLT with i.i.d data. Finally, note that one could also use the assumptions and result in Corollary~\ref{cor:clt-non-psd-cov} to relax the assumption made in Corollary~\ref{corr:lineareg} further and obtain a corresponding result. 

 \subsection{Stochastic Gradient Setting}\label{sec:sgdsetting}
 We now consider the stochastic gradient updates given in \eqref{eq:sgdupdate} and~\eqref{eq:noisygrad} for the optimization problem~\eqref{eq:main_prob}. The result of Theorem~\ref{thm:sgdsetting} below is analogous to that of Theorem~\ref{thm:linearsetting} for the linear setting.
 
 \begin{theorem}\label{thm:sgdsetting}
Consider optimizing~\eqref{eq:main_prob} using the iterative updates in~\eqref{eq:sgdupdate}. Let the function $f$ satisfy Assumption~\ref{assume:forsgd}. Then, for some universal constants $K, K_2, C',c_1 >0$, $c_2 \in (0,1)$, 
and for a twice differentiable function $h : \mathbb{R}^d \to \mathbb{R}$,
we have the following non-asymptotic bound, quantifying the rate of convergence of the iterates to a standard normal random vector $Z$, 
\begin{align}\label{eq:sgd-setting-bound}
  & \E\left[|h(\Sigma_t^{-\nhalf}~\bar{\Delta}_t) - h(Z)|\right]\\
  \nonumber
    &\leq  \frac{3\pi}{8t} \sqrt{d} M_2(h)
    \sum_{k=1}^t \E\left[ \big\|\Sigma_t^{\nhalf}  P_k^{-1} \Sigma_t^{\nhalf} \big\|^{\nhalf}_2
      \big\| \left([\nabla^2 f(\theta^*)]^{-1}\Sigma_t[\nabla^2 f(\theta^*)]^{-1} \right)^{-\nhalf} X_k \big\|_2^3 \right]\\
  \nonumber
  & \quad + \frac{M_1(h)\|\Sigma_t^{-\nhalf} \|_2}{t} \left[\frac{K_2\left[ \E\|\Delta_0\|_2\right]}{\eta_0} +\frac{KL_H \sum_{j=1}^{t-1} \sqrt{\eta_j}}{\sqrt{2\mu}} +\sqrt{K_d~K \varrho(\eta,t)}  \right]\\
  \nonumber
& \quad + \frac{3M_2(h)\|\Sigma_t^{-\nhalf} \|_2^2}{2t^2} \left[\frac{K_2^2\left[ \E\|\Delta_0\|_2^2\right]}{\eta_0^2} + \frac{K^2 L_H^2 \sum_{j=1}^{t-1} \eta_j}{2\mu}+  K_d~K \varrho(\eta,t)\right],
\end{align}
where $\varrho(\eta,t)$ is defined in~\eqref{eq:temp}, $X_k \defeq \left[ \E\left[\nabla f(\theta_{k-1})\right]  - \nabla f(\theta_{k-1})- \zeta_k\right]$ and $\Sigma_t \defeq \sum_{k=1}^t V_k$ where $V_k$ corresponds to the covariance matrix of $X_k$:
\begin{align*}
V_k \aseq \E \left[X_k X_k^\top|\mathcal{F}_{k-1}  \right], \qquad \forall k \in [t].
\end{align*}
\end{theorem}
We make several remarks about Theorem~\ref{thm:sgdsetting} and the bound \eqref{eq:sgd-setting-bound}. Notice the following difference between the Theorem~\ref{thm:linearsetting} for the linear setting and Theorem~\ref{thm:sgdsetting}. First note that the scaling of $\bar{\Delta}_t$ for Theorem~\ref{thm:linearsetting} is simplified and reads as $\sqrt{t}$ because of our assumption on the covariance matrix of the martingale. In Theorem~\ref{thm:sgdsetting}, the scaling is by $\Sigma_t^{-\nhalf}$. This leads to the rates controlled by $\| \Sigma_t^{-\nhalf}\|_2/t$ and $\| \Sigma_t^{-\nhalf}\|_2^2/t^2$ for the second and third term respectively. Next note that in~\cite{polyak1992acceleration}, a further assumption is made that the martingale $X_t$ could be decomposed as a summation of two term $X_t = X_t^0 + X_t^1$ such that 
 \begin{align*}
\E\left[ X_t^0  [X_t^0]^\top | \mathcal{F}_{t-1}\right]  \aseq \E \left[ \zeta_1 \zeta_1^\top \right] \defeq V.
 \end{align*}
 Furthermore, assumptions are made on the covariance matrix of the $X_t^1$ and the cross covariance matrix between $X_t^0$ and $X_t^1$ so that the covariance of $X_t$ is asymptotically almost surely equal to the covariance of $X_t^0$. In Theorem~\ref{thm:sgdsetting}, one could potentially make similar non-asymptotic  assumptions to decompose the first term in the R.H.S of the upper bound as sum of two terms that mimic the asymptotic covariance from~\cite{polyak1992acceleration}.  

\section{Discussion}
\label{sec:conclusion}
In this paper, we establish non-asymptotic convergence rates to normality of the Polyak-Ruppert averaged SGD algorithm. A main ingredient for establishing such a result is our result on convergence rates of certain martingale CLTs, which might be of independent interest. Our results have interesting consequences for confidence intervals and hypothesis tests based on SGD algorithm, that are justifiable in a non-asymptotic sense. A straightforward approach would be to leverage the bootstrap style algorithms proposed in~\cite{fang2018online} or~\cite{su2018statistical}, and apply our non-asymptotic martingale CLT result. Indeed a detailed look at the proof in the above papers, reveal that their proof strategy is very similar to the analysis of~\cite{polyak1992acceleration}, and the only probabilistic result used is the asymptotic martingale CLT result. Their results can be adapted to the non-asymptotic regime using our non-asymptotic martingale CLT result in Theorem~\ref{thm:martingale-clt}.

Furthermore, Berry-Esseen style results could be obtained using our  Theorem~\ref{thm:martingale-clt}, for the case of convex confidence sets following the idea of ``differentiable functions approximating indicator functions of convex sets" as done in~\cite{bentkus2003dependence, bentkus2005lyapunov}. Such a method is now standard for the i.i.d case. It is also worth emphasizing that the geometry of the confidence set plays a crucial role for general test functions. This is an active area of research, even for the i.i.d. setting; see, for example,~\cite{chernozhukov2017central}. Our results in this paper form an important step toward this in direction for the martingale case (and hence SGD inference).  

There are several extensions possible for future work. First, our current results for martingale CLT are established for the class of twice-differentiable functions. It would be interesting to extend such results for other metrics, for example, Wasserstein or Kolmogorov distance. Next, it would be interesting to establish non-asymptotic rates of convergence to normality for non-smooth optimization problems and other variants of SGD algorithm.
\bibliographystyle{amsalpha}
\bibliography{./bib}
\appendix
%

\section{Proofs of the Martingale CLT Results}
\label{sec:proof-clt}

The proof of Theorem~\ref{thm:martingale-clt} relies on a combination of Stein's method and Lindeberg's telescopic sum argument as first outlined by \cite{rollin2018quantitative}
for real valued martingales. 
We start with the following lemma which is standard in the literature of Stein's method,
and follows from similar steps already carried out in \cite{goldstein1996multivariate,raivc2004multivariate,gaunt2016rates}.
\begin{lemma} \label{lem:stein-factor-bound}
  Let $h : \reals^d \to \reals$ be twice differentiable and $Z\sim \Gsn_d(0, I)$.
  For $\mu \in \reals^d$ and $\Sigma \in \reals^{d \times d}$ symmetric and positive definite,
  there is a solution $f : \reals^d \to \reals$ to the Stein equation
  \eq{\label{eq:stein-eq}
    \inner{\Sigma,\Hess f(x)} - \inner{x-\mu, \grad f(x)}
    =h(x) - \E[h(\Sigma^{\nhalf}Z+\mu)],
  }
  where we have
  \eq{
    M_3(f) 
    \leq  M_2(h) \frac{\pi}{4}\sqrt{d}
    \|\Sigma^{-\nhalf}\|_2.
  }
\end{lemma}
The bound $M_i(f)$ on the $i$-th order derivative of $f$ are termed as the $i$-th Stein factor. 
The main important property of Lemma~\ref{lem:stein-factor-bound} is that
it connects the third Stein factor to the smoothness of the test function $h$.
In the one-dimensional case, one can obtain a bound on third Stein factor in terms of
the Lipschitz constant
of the test function \cite{stein1986approximate,rollin2018quantitative}
without any smoothness assumptions;
however, at least polynomial smoothness is required
in the higher dimensions \cite{gorham2016measuring,erdogdu2018nonconvex}.
The proof for the above result is given below for reader's convenience.\\
%

\begin{proof}[Proof of Lemma~\ref{lem:stein-factor-bound}]
  For a twice differentiable test function $\th : \reals^d \to \reals$,
  and choosing $\mu=0$ and $\Sigma=I$,
  the Stein equation \eqref{eq:stein-eq} reduces to
  \eq{\label{eq:stein-eq-reduced}
    \Tr(\Hess \tf(x)) - \inner{x, \grad \tf(x)}
    =\th(x) - \E[\th(Z)].
  }
  It is well known that the function $\tf$ that solves the Stein equation
  is given by
  \eq{
    \tf(x) = \int_0^\infty \E[\th(Z)] - (P_t\th)(x)  \dt,
  }
  where
  $(P_t\th)(x) = \Ex{\th(e^{-t} x + \sqrt{1 - e^{-2t}}Z)}$ and $(P_t)_{t\geq 0}$
  is the semigroup associated with the Ornstein-Uhlenbeck diffusion \cite{barbour1990stein}.
  One can easily verify that \eqref{eq:stein-eq} is recovered by the following
  change of variables
  \eq{
    h(x) = \th(\Sigma^{-\nhalf}(x-\mu))\ \text{ and }\ f(x) = \tf(\Sigma^{-\nhalf}(x-\mu)),
  }
  where $f(x)$ is given by
  \eq{
    f(x) =& \int_0^\infty \E[\th(Z)] - (P_t\th)(\Sigma^{-\nhalf}(x-\mu)) \dt,\\
    =& \int_0^\infty \E[h(\Sigma^{\nhalf}Z+\mu)] -
    \Ex{h(e^{-t} (x-\mu) + \sqrt{1 - e^{-2t}}\Sigma^{\nhalf}Z + \mu)} {\mathrm d}t.
  }
  Denoting the $d$-variate istoropic Gaussian density with $\phi_d(z)$,
  and following a similar approach with
  \cite{raivc2004multivariate,goldstein1996multivariate},
  we take a derivative in the unit direction $v\in\reals^d$ (i.e. $\twonorm{v}=1$),
  and apply integration by parts to obtain
  \eq{
    \inner{\grad f(x),v}=& -\int_0^\infty e^{-t}
    \Ex{\inner{\grad h(e^{-t} (x-\mu) + \sqrt{1 - e^{-2t}}\Sigma^{\nhalf}Z + \mu),v}} {\mathrm d}t,\\
    \nonumber
    =& -\int_0^\infty \frac{e^{-t}}{\sqrt{1 - e^{-2t}}} \int
    h(e^{-t} (x-\mu) + \sqrt{1 - e^{-2t}}\Sigma^{\nhalf}z + \mu)
    \inner{\grad \phi_d(z), \Sigma^{-\nhalf}v}{\mathrm d}z {\mathrm d}t.
  }

  Taking one more derivative in the unit direction of $w \in \reals^d$ yields
  \eqn{
    \inner{\Hess f(x)v,w}
    =& -\int_0^\infty \frac{e^{-2t}}{\sqrt{1 - e^{-2t}}} \int
    \inner{\grad h(e^{-t} (x-\mu) + \sqrt{1 - e^{-2t}}\Sigma^{\nhalf}z + \mu),w}
    \inner{\grad \phi_d(z), \Sigma^{-\nhalf}v}{\mathrm d}z {\mathrm d}t.
  }
  Using the smoothness properties of the test function $h$ and $\twonorm{w}=1$, we can write
  \eq{
    &\abs{\inner{\grad h(e^{-t} (x-\mu) + \sqrt{1 - e^{-2t}}\Sigma^{\nhalf}z + \mu)- \grad h(e^{-t} (y-\mu) + \sqrt{1 - e^{-2t}}\Sigma^{\nhalf}z + \mu),w}}\\
    &\leq M_2(h) e^{-t} \twonorm{x - y}.
  }
  
  Using this we can bound the third Stein factor as follows,
  \eqn{
    \frac{\abs{\inner{\Hess f(x)v,w} - \inner{\Hess f(y)v,w}}}{\twonorm{x-y}} \leq&
    M_2(h)\int_0^\infty \frac{e^{-3t}}{\sqrt{1 - e^{-2t}}}\dt
    \int\abs{\inner{\grad \phi_d(z), \Sigma^{-\nhalf}v}}\dz,\\
    \leq&M_2(h)\frac{\pi}{4}\|\Sigma^{-\nhalf}\|_2
    \int\|z\|_2 \phi_d(z)\dz,\\
    =&M_2(h)\frac{\pi}{2\sqrt{2}}\frac{\Gamma(\tfrac{d+1}{2})}{\Gamma(\tfrac{d}{2})}
    \|\Sigma^{-\nhalf}\|_2,\\
    \leq&M_2(h)\frac{\pi}{4}\sqrt{d}
    \|\Sigma^{-\nhalf}\|_2,
  }
  which completes the proof.
\end{proof}

Before we prove Theorem~\ref{thm:martingale-clt},
we define several useful sequences related to the martingale difference sequence $\{X_k\}_{k=1}^n$.
Recalling the basic properties of the martingale difference sequences,
we have
\eq{\label{eq:martingale-difference}
  \Ex{X_k| \F_{k-1}} = 0 \ \text{ and }\ \s_k = \E[X_kX_k^\T|\F_{k-1}]\ \text{ almost surely,}
}
where $0$ is understood to be a vector of zeros in $\reals^d$, and $\s_k$ is a random matrix in $\reals^{d \times d}$.
For $k\in [n]$, we define the partial sums as
$S_k = X_k + S_{k-1}$ starting from $S_0 = 0$,
and the partial covariances as
\eq{\label{eq:cond-covariances}
  \ts_k = \sum_{i=1}^k\s_i, \ \ P_{k+1} = \ts_n - \ts_k, \ \ \bsig_k=\E[\s_k],\
  \text{ and } \ \vn =  \sum_{i=1}^n \bsig_i.
}
We note that using the above definitions, we have $\Var(S_n)= \Sigma $.

 \begin{proof}[Proof of Theorem~\ref{thm:martingale-clt}]
 
  Let $Z',Z_1,...,Z_n $ be a sequence of independent isotropic normal random vectors,
  also independent of $\F_k$ for $k\in [n]$. Define
  \eq{
    Z = \sum_{i=1}^n \s_i^{\nhalf}Z_i,\ \ T_k = \sum_{i=k}^n \s_i^{\nhalf}Z_i
  }
  with $T_1 = Z$ and $T_{n+1} = 0$. Observe that $Z\sim \Gsn(0,V_n)$ since $\ts_n = \vn$
  almost surely under Assumption~\ref{asmp:almost-sure}. We also note that
  $\ts_k \in \F_{k-1}$; thus, $P_k$ and $P_{k+1}\in \F_{k-1}$. This implies
  \eq{\label{Tk}
    T_k | \F_{k-1} \sim \Gsn(0, P_k)\ \text{ and }\   T_{k+1} | \F_{k-1} \sim \Gsn(0, P_{k+1}).
  }

  Define the random variable $R_k := h(S_k + \Tk_{k+1}) - h(S_{k-1} + \Tk_k)$,
  and observe that
  \eq{\label{eq:telescop}
    h(S_n) - h(Z) = \sum_{i=1}^n R_i.
  }
  For $Y\sim \Gsn(0,I)$ independent from all random variables and the filtration, we write
  \eq{
    \E[R_k|\F_{k-1}] = &\Ex{h(S_k + \Tk_{k+1}) - h(S_{k-1}+\Tk_k)| \F_{k-1}},\\
    = &\Ex{h(S_k + \Tk_{k+1}) - h(S_{k-1}+P_k^{\nhalf}Y)| \F_{k-1}},
  }
  where $P_k$ is as in \eqref{Tk}. Using now the Stein equation as in \eqref{eq:stein-eq}, we obtain
  \eq{
    \E[R_k|\F_{k-1}]
    = \E[\inner{P_k,\Hess f_k(S_k + \Tk_{k+1})}
    - \inner{X_k + \Tk_{k+1}, \grad f_k(S_k + \Tk_{k+1})}|\F_{k-1}].
  }
  Observing that $S_k$ and $T_{k+1}$ are independent conditional on $\F_{k-1}$,
  and $T_{k+1}|\F_{k-1} \sim \Gsn(0, P_{k+1})$
  and applying Lemma \ref{lem:stein-factor-bound} to the right hand side above, we get
  \eqn{
    \E[R_k|\F_{k-1}]
    =& \E[\inner{P_k,\Hess f_k(S_k + \Tk_{k+1})}
    - \inner{X_k , \grad f_k(S_k + \Tk_{k+1})} - \inner{P_{k+1}, \Hess f_k(S_k + \Tk_{k+1})}|\F_{k-1}],\\
    =&\E[\inner{\Sigma_k,\Hess f_k(S_k + \Tk_{k+1})} - \inner{X_k , \grad f_k(S_k + \Tk_{k+1})}| \F_{k-1}].
  }
  For the first term above, we write
  \eq{
    \E[\inner{\Sigma_k, \Hess f_k(S_k + \Tk_{k+1})}| \F_{k-1}]
    =&  \int_0^1\E[\inner{\Sigma_k,
      \grad^3 f_k(S_{k-1} sX_k+ \Tk_{k+1})}[X_k] | \F_{k-1}]{\mathrm d}s\\
    \nonumber
    &+\E[\inner{\Sigma_k, \Hess f_k(S_{k-1} + \Tk_{k+1})}| \F_{k-1}]
  }
  and for the second term, we use Taylor's theorem and obtain
  \eq{
    \E[\inner{X_k, \grad f_k(S_k + \Tk_{k+1})| \F_{k-1}}]
    &=\E[\inner{X_k, \grad f_k(S_{k-1}+\Tk_{k+1})}|\F_{k-1}]\\
    \nonumber
    +&\E[\inner{X_k, \Hess f_k(S_{k-1}+\Tk_{k+1})X_k}|\F_{k-1}]\\
    \nonumber
    +&\int_0^1(1-s)\E[\inner{X_k, \grad^3 f_k(S_{k-1} + sX_k + \Tk_{k+1})[X_k,X_k]}|\F_{k-1}]{\mathrm d}s.
  }

  Combining these, we obtain
  \eq{
    \abs{\Ex{R_k|\F_{k-1}}}=&
    \Bigg|\int_0^1\E[\inner{\Sigma_k,
      \grad^3 f_k(S_{k-1} + sX_k+ \Tk_{k+1})[X_k]} | \F_{k-1}]{\mathrm d}s\\
    \nonumber
    &-\int_0^1(1-s)\E[\inner{X_k, \grad^3 f_k(S_{k-1} + sX_k + \Tk_{k+1})[X_k,X_k]}
    |\F_{k-1}]{\mathrm d}s\Bigg|
  }
  For the first term above, we have
  \eq{
    \E[\inner{\Sigma_k,
      \grad^3 f_k(S_{k-1} + sX_k+ \Tk_{k+1})[X_k]} | \F_{k-1}]
    \leq& M_3(f)\E[\Tr(\Sigma_k)
    \twonorm{X_k} | \F_{k-1}],\\
    \nonumber
    \leq& M_3(f_k)\E[ \twonorm{X_k}^3 | \F_{k-1}],
  }
  and similarly for the second term, we have
  \eq{
    \E[\inner{X_k, \grad^3 f_k(S_{k-1} + sX_k + \Tk_{k+1})[X_k,X_k]}
    |\F_{k-1}] \leq M_3(f_k)\E[ \twonorm{X_k}^3 | \F_{k-1}].
  }

  Note that $M_3(f_k)$ is $\F_{k-1}$-measurable.
  Iterating this bound in \eqref{eq:telescop} and taking expectations on both sides,
  we obtain
  \eq{
    \abs{\Ex{h(S_n) - \Ex{h(Z)}}} \leq& \sum_{k=1}^n 1.5 \  \Ex{M_3(f_k)\twonorm{X_k}^3 | \F_{k-1}}.
  }
  Scaling with $\vn^{-\nhalf}$ together with Lemma~\ref{lem:stein-factor-bound}, we obtain
  \eq{
    \abs{\Ex{h(\vn^{-\nhalf}S_n) - \Ex{h(Z)}}}
    \leq \frac{3\pi}{8}\sqrt{d}
    M_2(h)\sum_{k=1}^n \Ex{\|P_k^{-\nhalf}\vn^{\nhalf}\|_2
      \|\vn^{-\nhalf}X_k\|_2^3}.}
\end{proof}

\begin{proof}[Proof of Corollary~\ref{cor:clt-psd-cov}]

  The assumptions in Corollary~\ref{cor:clt-psd-cov} imply
  $n\beta I \succeq \vn \succeq n\alpha I$
  and $(n-k)\beta I \succeq P_{k+1} \succeq (n-k)\alpha I$ almost surely.
  Plugging these in the bound in Theorem~\ref{thm:martingale-clt}, we obtain
  \eq{
    \abs{\Ex{h(\vn^{-\nhalf}S_n) - \Ex{h(Z)}}} &\leq 
    \frac{3\pi}{8}\sqrt{d}
    M_2(h)\sum_{k=1}^n \Ex{\|P_k^{-\nhalf}\vn^{\nhalf}\|_2
      \|\vn^{-\nhalf}X_k\|_2^3},\\
    &\leq \frac{3\pi}{8}
    M_2(h)\frac{\gamma\sqrt{\beta}d^2}{\alpha^{2} n}\sum_{k=1}^n \frac{1}{\sqrt{n-k+1}}.
  }
  Finally, we notice that $\sum_{k=1}^n \tfrac{1}{\sqrt{n-k+1}} = \sum_{k=1}^n \tfrac{1}{\sqrt{k}} < 2\sqrt{n}$ which concludes the proof.
\end{proof}

\begin{proof}[Proof of Corollary~\ref{cor:clt-non-psd-cov}]
  For a given matrix $A \in \reals^{d \times d}$,
  define the random variable $\Tkp_k = \Tk_k + AZ'$ where $\Tk_k$ is as in the proof of Theorem~\ref{thm:martingale-clt},
  and let $$R_k \defeq h(S_k + \Tkp_{k+1}) - h(S_{k-1} + \Tkp_k).$$
  Following the same steps as in the proof of Theorem~\ref{thm:martingale-clt} where $\Tk_k$ replaced
  with $\Tkp_k$, we obtain
  \eq{
    &\abs{\Ex{h(\vn^{-\nhalf} S_n) - \Ex{h(Z)}}}\\
    &\leq  \frac{3\pi}{8}\sqrt{d}
    M_2(h)\sum_{k=1}^n
    \Ex{\|(\vn^{\nhalf} (P_k +A\vn^{-1}A^\T)^{-1} \vn^{\nhalf}\|_2^{\nhalf}
      \|\vn^{-\nhalf}X_k\|_2^3} + 2M_1(h)\Tr(A\vn^{-1}A^\T)^{\nhalf} .
  }

  Next, we define the stopping times,
  \eq{
    \tau_0 = 0, \ \ \tau_k = \sup \{m\geq 0 \ : \ \ts_m \preceq \tfrac{k}{n} \vn  \}\ \text{ for }\ 1 \leq k \leq n.
  }
  Since $\{ \tau_k = m\} = \{ \ts_m \preceq \tfrac{k}{n} \vn \} \cap \{\ts_{m+1} \succ \tfrac{k}{n} \vn \}$,
  and each of these events are $\F_m$-measurable, $\tau_k$ is a stopping time for each $k$.
  If $j \leq \tau_k$, we can write
  \eq{
    P_j = \ts_n - \ts_{j-1} = \vn - \ts_{j-1} \succeq \tfrac{n-k}{n} \vn .
  }
  Therefore we can write,
  \eq{
    &\Ex{
      \sum_{j=\tau_{k-1}+1}^{\tau_k}
      \|\vn^{\nhalf} (P_j + \vn/n)^{-1} \vn^{\nhalf} \|_2^{\nhalf}
      \|\vn^{-\nhalf} X_j\|_2^3 
    }\\
    &=\Ex{
      \sum_{j=1}^{n}
      \Ex{
        \|\vn^{\nhalf} (P_j + \vn/n)^{-1} \vn^{\nhalf} \|_2^{\nhalf}
        \|\vn^{-\nhalf} X_j\|_2^3 \ind{\tau_{k-1} < j \leq \tau_k}
        | \F_{j-1}
      }
    }, \\
    &\leq
    \sqrt{\frac{n}{n-k+1}}
    \Ex{
      \sum_{j=1}^{n}
      \ind{\tau_{k-1} < j \leq \tau_k}
      \Ex{
        \|\vn^{-\nhalf} X_j\|_2^3 
        | \F_{j-1}
      }
    },\\
    &\leq
    \delta\|\vn^{-\nhalf}\|^3_2    \sqrt{\frac{n}{n-k+1}}
    \Ex{
      \sum_{j=1}^{n}
      \ind{\tau_{k-1} < j \leq \tau_k}
      \Tr(\s_j)
    },\\
    &\leq
    \delta\|\vn^{-\nhalf}\|^3_2   \sqrt{\frac{n}{n-k+1}}
    \Ex{
      \Tr(\ts_{\tau_k} - \ts_{\tau_{k-1}})
    },\\
    &\leq
    \delta\|\vn^{-\nhalf}\|^3_2   \sqrt{\frac{n}{n-k+1}}
    ( \tfrac{1}{n}\Tr(\vn) + \beta^{2/3}).
  }

  Next, we sum over $k$ and obtain
  \eq{
    &\sum_{k=1}^{n}  \Ex{
      \|\vn^{\nhalf} (P_k + \vn/n)^{-1} \vn^{\nhalf} \|_2^{\nhalf}
      \|\vn^{-\nhalf} X_k\|_2^3 
    }, \\
    &\leq \delta\|\vn^{-\nhalf}\|^3_2
    ( \tfrac{1}{n}\Tr(\vn) + \beta^{2/3})
    \sum_{k=1}^{n}
    \sqrt{\frac{n}{k}},\\
    &\leq 2\delta\|\vn^{-\nhalf}\|^3_2
    ( \tfrac{1}{n}\Tr(\vn) + \beta^{2/3})n.
  }

  Therefore, choosing $A = \vn/\sqrt{n}$ we obtain
  \eq{
    &\abs{\Ex{h(\vn^{-\nhalf} S_n) - \Ex{h(Z)}}}\\
    &\leq  \frac{3\pi}{4}\sqrt{d}
    M_2(h) \delta\|\vn^{-\nhalf}\|^3_2
    \big\{ \tfrac{1}{n}\Tr(\vn) + \beta^{2/3} \big \}n
    + \frac{2M_1(h)}{\sqrt{n}}\Tr(\tfrac{1}{n}\vn)^{\nhalf} .
  }
\end{proof}

\section{Proofs for Section~\ref{sec:linearsetting}}
\begin{proof}[Proof of Theorem~\ref{thm:linearsetting}]
The proof involves adapting the proof of Theorem in~\cite{polyak1992acceleration} to our non-asymptotic setting and then applying our result from Theorem~\ref{thm:martingale-clt}. To proceed, define 
\begin{align*}
B_j^t = \eta_j \sum_{i=j}^{t-1} \prod_{k=j+1}^i [I - \eta_k A]
\end{align*}
and set $B_t = B_0^t$ and $W_j^t = B_j^t - A^{-1}$. Then, by Lemma 2 in~\cite{polyak1992acceleration}, for the above sequence of matrices $B_t$, $t \in \mathbb{N}$ and the triangular array of matrices $W^t_j$, $t \in \mathbb{N}$ and $j \leq t$, we have the following decomposition for $\bar{\Delta}_t $:
\begin{equation}
\label{Delta_notation}
\sqrt{t} \bar{\Delta}_t = \underbrace{\frac{1}{\sqrt{t}\eta_0} B_t \Delta_0}_{I_1} + \underbrace{\frac{1}{\sqrt{t}} \sum_{j=1}^{t-1}A^{-1}\zeta_j}_{I_2} + \underbrace{\frac{1}{\sqrt{t}} \sum_{j=1}^{t-1} W^t_j \zeta_j}_{I_3}.
\end{equation}
In addition, we have $\|B_t\|_2 \leq K_2 < \infty$. We now proceed to provide a non-asymptotic result about the rate of convergence of the iterates to normality. Recall that, with $Z$ denoting an isotropic normal random vector, we are looking for a non-asymptotic bound on $\E|h(\sqrt{t} \bar{\Delta}_t) - h(Z)|$. Using the triangle inequality,
\begin{align}
\label{term1_to_bound}
\left|\E\left[h(\sqrt{t} \bar{\Delta}_t)\right] - \E\left[h(A^{-1}V^{1/2}Z)\right]\right| & \leq \left|\E\left[h(I_2)\right] - \E\left[h(A^{-1}V^{1/2}Z)\right]\right|\\
\label{term2_to_bound}
& \quad + \left|\E \left[h\left(\sqrt{t} \bar{\Delta}_t\right) - h(I_2)\right]\right|,
\end{align}
where $I_2$ is as in \eqref{Delta_notation}.\\
{\textbf{Bound for \eqref{term1_to_bound}}}. Note that $I_2$ forms a martingale difference sequence. Hence, we use Theorem~\ref{thm:martingale-clt} to handle this term.  Under our assumption, first note that we have
\begin{align*}
P_{k} = (t-(k-1)) A^{-1}V A^{-1}\qquad \Sigma_t = t A^{-1}V A^{-1}.
\end{align*}
Therefore,
\begin{align}\label{eq:termi2}
\nonumber \text{R.H.S. of}~\eqref{term1_to_bound} & =\leq \frac{3\pi}{8} \sqrt{d} M_2(h) \sum_{k=1}^t \E \left[ \sqrt{\frac{t}{(t-(k-1))}} ~\left\| \frac{ (A^{-1} V A^{-1})^{-1/2}\zeta_{k}}{t^{1/2}}\right\|_2^3\right] \\
 & \leq \sum_{k=1}^t \E \left[ \frac{ 1.18 \sqrt{d} M_2(h)~\left\|  \left[A^{-1} V A^{-1}\right]^{-1/2}\zeta_{k}\right\|_2^3}{t~\sqrt{(t-(k-1))}} \right] 
\end{align}
{\textbf{Bound for \eqref{term2_to_bound}}}. A second-order Taylor expansion of $h(I_1+ I_2 +I_3)$ around $I_2$ yields
\begin{equation}
\label{Taylor_expansion}
h\left(I_1 + I_2 + I_3\right) = h\left(I_2\right) + \langle\nabla h(I_2), I_1 + I_3\rangle + \frac{1}{2}\langle \nabla^2 h(I_2 + \bar{c}(I_1 + I_3))(I_1 + I_3), I_1 + I_3\rangle,
\end{equation}
where the last term on the right-hand side of \eqref{Taylor_expansion} is the remainder term expressed in Lagrange's form for $\bar{c}$, a positive constant in $(0,1)$. The result in \eqref{Taylor_expansion} and the triangle inequality lead to
\begin{align}
\label{bound_term2_general}
\nonumber \eqref{term2_to_bound} & = \left|\E\left[\langle\nabla h(I_2), I_1 + I_3\rangle + \frac{1}{2}\langle \nabla^2 h(I_2 + \bar{c}(I_1 + I_3))(I_1 + I_3), I_1 + I_3\rangle\right]\right|\\
\nonumber & \leq  \E \left|\langle \nabla h(I_2), (I_1+I_3)\rangle\right| + \frac{1}{2}\E\left|\langle \nabla^2 h(I_2 + \bar{c}(I_1 + I_3))(I_1 + I_3), I_1 + I_3\rangle\right|\\ 
\nonumber & \leq  M_1(h)\E\left[\|I_1+I_3\|_2\right] + \frac{1}{2}\E\left[\|\nabla^2 h(I_2 + \bar{c}(I_1 + I_3))(I_1 + I_3)\|_2\|I_1 + I_3\|_2\right]\\
\nonumber & \leq M_1(h)\E\left[\|I_1+I_3\|_2\right] + \frac{M_2(h)}{2}\E\left[\|I_1 + I_3\|^2_2\right]\\
& \leq M_1(h)\E\left[\|I_1\|_2+\|I_3\|_2\right] + M_2(h)\E\left[\|I_1\|^2_2 + \|I_3\|^2_2\right].
 \end{align}
First, we could upper bound $\E\left[\|I_1\|_2\right]$ as follows.
\begin{equation}\label{eq:termi1I_1}
\E\|I_1\|_2 =  \frac{1}{\sqrt{t}\eta_0}  \E \left[ \left\|B_t \Delta_0 \right\|_2\right] \leq  \frac{1}{\sqrt{t}\eta_0}  \E \left[ \|B_t\|_2 \|\Delta_0\|_2\right ] =  \frac{K_2\left[ \E\|\Delta_0\|_2\right]}{\sqrt{t}\eta_0}
\end{equation}
Following the same approach, it is straightforward that
\begin{equation}
\label{eq:termi1I_1_squared}
\E\left[\|I_1\|_2^2\right] \leq \frac{K_2^2\left[ \E\|\Delta_0\|_2^2\right]}{t\eta_0^2}.
\end{equation}
Continuing now to find an upper bound for $\E\left[\|I_3\|_2\right]$ and $\E\left[\|I_3\|_2^2\right]$, note that by definition we have
\begin{align*}
W_j^t = A^{-1} T_j^t +  S_j^t,
\end{align*}
where $T_j^t =\prod_{k=j+1}^i [I - \eta_k A] $ and $S_j^t = \eta_j\sum_{i=j}^{t-1}T_j^t $ . Hence, we have
\begin{align*}
\frac{1}{t} \sum_{j=1}^{t-1} \| W_j^t\|_2^2 & =\frac{1}{t} \sum_{j=1}^{t-1} \| A^{-1} T_j^t +  S_j^t \|_2^2 \\
&\leq\frac{2}{t} \sum_{j=1}^{t-1} \left[  \| A^{-1} T_j^t\|_2 + \| S_j^t\|_2\right]^2 \leq \frac{4}{t} \sum_{j=1}^{t-1} \left[\| A^{-1} T_j^t\|_2^2 + \| S_j^t\|_2^2\right]\\
&=\frac{4C}{t} \sum_{j=1}^{t-1} \left[  \| T_j^t\|_2^2 \right]+ \frac{4}{t} \sum_{j=1}^{t-1}\left[ \| S_j^t\|_2^2\right]
\end{align*}
Now, we have by Lemma 1, Part 3 in~\cite{polyak1992acceleration}
\begin{align*}
\| T_j^t\|_2^2  \leq K \left[ e^{\left( -c_1 \sum_{i=j}^{t-1} \eta_i \right)  }\right]^2
 \end{align*}
 Furthermore, we have from page 846 in~\cite{polyak1992acceleration}, the following estimate for $ \| S_j^t\|^2_2 $, with $m_j^i \defeq \sum_{k=j}^i \eta_k$ and some $c_2 >0$,
\begin{align*}
 \| S_j^t\|_2^2 \leq \left[\frac{C'}{\eta_j^{1-c_2}} \sum_{i=j}^t m_j^t e^{-\lambda m_j^t} (m_j^i -m_j^{i-1})\right]^2
 \end{align*}
Combining all the above estimates, we have for some constant $K$, 
\begin{align*}
\frac{1}{t} \sum_{j=1}^{t-1} \| W_j^t\|_2^2 \leq  \frac{K}{t} \sum_{j=1}^{t-1} \left\{  \left[ e^{\left( -c_1 \sum_{i=j}^{t-1} \eta_i   \right)  }\right]^2   + \left[\frac{C'}{\eta_j^{1-c_2}} \sum_{i=j}^t m_j^t e^{-\lambda m_j^t} (m_j^i -m_j^{i-1})\right]^2 \right\}
\end{align*}
Using the above result and the fact that $\E\left[\zeta_i^{\intercal}\zeta_j\right] = 0$ for $i \neq j$, we have that
\begin{align}\label{eq:termi3_squared}
\nonumber \E \left[\|I_3\|_2^2\right] & = \frac{1}{t}\E\left[ \left \|\sum_{j=1}^{t-1} W^t_j \zeta_j\right \|_2^2\right] \leq K_d \left[\frac{1}{t} \sum_{j=1}^{t-1} \| W_j^t\|_2^2\right]\\
& \leq \frac{K_d~K}{t} \sum_{j=1}^{t-1} \left\{  \left[ e^{\left( -c_1 \sum_{i=j}^{t-1} \eta_i   \right)  }\right]^2   + \left[\frac{C'}{\eta_j^{1-c_2}} \sum_{i=j}^t m_j^t e^{-\lambda m_j^t} (m_j^i -m_j^{i-1})\right]^2 \right\}
\end{align}
Since $\E[\|I_3\|_2] \leq  \left[\E \left[\|I_3\|_2^2 \right]\right]^{1/2}$, then using the results in \eqref{eq:termi1I_1}, \eqref{eq:termi1I_1_squared}, \eqref{eq:termi3_squared}, and continuing from \eqref{bound_term2_general}, we have that
\begin{align}
\label{bound2}
\nonumber \eqref{term2_to_bound} \leq & \frac{M_1(h)}{\sqrt{t}}\left[\frac{K_2\left[ \E\|\Delta_0\|_2\right]}{\eta_0} + \sqrt{K_d~K \sum_{j=1}^{t-1} \left\{  \left[ e^{\left( -c_1 \sum_{i=j}^{t-1} \eta_i   \right)  }\right]^2   + \left[\frac{C'}{\eta_j^{1-c_2}} \sum_{i=j}^t m_j^t e^{-\lambda m_j^t} (m_j^i -m_j^{i-1})\right]^2 \right\}   }\right]\\
& + \frac{M_2(h)}{t}\left[\frac{K_2^2\left[ \E\|\Delta_0\|_2^2\right]}{\eta_0^2} + K_d~K \sum_{j=1}^{t-1} \left\{  \left[ e^{\left( -c_1 \sum_{i=j}^{t-1} \eta_i   \right)  }\right]^2   + \left[\frac{C'}{\eta_j^{1-c_2}} \sum_{i=j}^t m_j^t e^{-\lambda m_j^t} (m_j^i -m_j^{i-1})\right]^2 \right\}\right].
\end{align}
The results now in \eqref{eq:termi2} and \eqref{bound2} give the desired result.
\end{proof}

\section{Proofs for Section~\ref{sec:sgdsetting}}\label{appendixend}

 \begin{proof}[Proof of Theorem~\ref{thm:sgdsetting}]
 To proceed, first note that~\eqref{eq:sgdupdate} and~\eqref{eq:noisygrad}  could be recast as follows:
 \begin{align*}
 \theta_t & = \theta_{t-1} -\eta_t \left[ \nabla f(\theta_{t-1}) + \zeta_t\right]+ \eta_t \E \left[\nabla f(\theta_{t-1})  \right] -   \eta_t \E\left[\nabla f(\theta_{t-1})  \right] \\
 & =\theta_{t-1} -\eta_t \E\left[\nabla f(\theta_{t-1})\right] + \eta_t \left[ \E\left[\nabla f(\theta_{t-1})\right]  - \nabla f(\theta_{t-1})- \zeta_t\right]\\
 &\defeq \theta_{t-1} - \eta_t \left[R(\theta_{t-1})\right] + \eta_t X_t\left(\theta_{t-1} -\theta^*\right).
  \end{align*}
Note that $X_t(\theta_{t-1} -\theta^*)$ forms a martingale difference sequence. Whenever there is no confusion, we just call it as $X_t$ in the rest of the proof. Recall that the covariance matrix of $X_t$ as
\begin{align*}
V_t \aseq \E \left[X_t X_t^\top|\mathcal{F}_{t-1}  \right],
\end{align*}
and note that $\Sigma_t \defeq \sum_{i=1}^t V_i$. Now, define
\begin{align*}
B_j^t = \eta_j \sum_{i=j}^{t-1} \prod_{k=j+1}^i [I - \eta_k \nabla^2f(\theta^*)], \qquad 
W_j^t = B_j^t - \left[\nabla^2 f(\theta^*)\right]^{-1},
\end{align*}
and set $B_t = B_0^t$. Then, in this case, we have the following decomposition similar to that of linear setting (specifically,~\eqref{Delta_notation}):
  \begin{align*}
 \Sigma_t^{-1/2} \bar{\Delta}_t & = \underbrace{\frac{2\Sigma_t^{-1/2}}{t~\eta_0} B_t \Delta_0}_{I_1} + \underbrace{\frac{\Sigma_t^{-1/2}}{t} \sum_{j=1}^{t-1}\nabla^2f(\theta^*)^{-1}X_j}_{I_2} + \underbrace{\frac{\Sigma_t^{-1/2}}{t} \sum_{j=1}^{t-1} W^t_j X_j}_{I_3}  \\ & \;\; +\underbrace{\frac{\Sigma_t^{-1/2}}{t} \sum_{j=1}^{t-1} [\left[\nabla^2f(\theta^*)\right]^{-1} + W^t_j] \left( \left[\bar{R}(\Delta_j) - \nabla^2f(\theta^*) \Delta_j\right]\right)  }_{I_4},
\end{align*}
where $\bar{R}(\theta) = R(\theta - \theta^*)$. First note that by triangle inequality, we have
\begin{align}
\label{eq:sgdterm1} \left|\E\left[h(\Sigma_t^{-1/2}\bar{\Delta}_t)\right] - \E\left[h(Z)\right]\right| & \leq \left|\E\left[h(I_2)\right] - \E\left[h(Z)\right]\right|\\
\label{eq:sgdterm2} 
& \quad + \left|\E \left[h\left(\Sigma_t^{-1/2} \bar{\Delta}_t\right) - h(I_2)\right]\right|.
\end{align}
{\textbf{Bound for \eqref{eq:sgdterm1}}}. Note that $I_2$ forms a martingale difference sequence. Hence, we use Theorem~\ref{thm:martingale-clt} to handle this term.  Under our assumption, similar to the proof of Theorem~\ref{thm:linearsetting}, we have
\begin{align*}
\nonumber \text{R.H.S. of}~\eqref{eq:sgdterm1} & \leq \frac{3\pi}{8t} \sqrt{d} M_2(h) \sum_{k=1}^t \E\left[ \left\|\Sigma_t^{1/2}  P_k^{-1} \Sigma_t^{1/2} \right\|^{1/2}_2 \left\| \left[[\nabla^2 f(\theta^*)]^{-1} ~\Sigma_t ~[\nabla^2 f(\theta^*)]^{-1} \right]^{-1/2} X_k \right\|_2^3 \right]
\end{align*}

\noindent {\textbf{Bound for \eqref{eq:sgdterm2}}}. To handle this term, we do a second-order taylor expansion of $h(I_1+ I_2 +I_3+I_4)$ around $I_2$, to get
\begin{align}
\label{sgd_Taylor_expansion}
h\left(I_1 + I_2 + I_3+I_4\right) &= h\left(I_2\right) + \langle\nabla h(I_2), I_1 + I_3+I_4\rangle \\
\nonumber &+  \frac{1}{2}\langle \nabla^2 h(I_2 + \bar{c}(I_1 + I_3+I_4))(I_1 + I_3+I_4), I_1 + I_3+I_4\rangle,
\end{align}
where the last term on the right-hand side of \eqref{sgd_Taylor_expansion} is the remainder term expressed in Lagrange's form for $\bar{c}$, a positive constant in $(0,1)$. The result in \eqref{sgd_Taylor_expansion} and the triangle inequality lead to
\begin{align}
\nonumber \eqref{eq:sgdterm2} & = \left|\E\left[\langle\nabla h(I_2), I_1 + I_3+I_4\rangle + \frac{1}{2}\langle \nabla^2 h(I_2 + \bar{c}(I_1 + I_3+I_4))(I_1 + I_3+I_4), I_1 + I_3+I_4\rangle\right]\right|\\
\nonumber & \leq  \E \left|\langle \nabla h(I_2), (I_1+I_3+I_4)\rangle\right| + \frac{1}{2}\E\left|\langle \nabla^2 h(I_2 + \bar{c}(I_1 + I_3+I_4))(I_1 + I_3+I_4), I_1 + I_3+I_4\rangle\right|\\ 
\nonumber & \leq  M_1(h)\E\left[\|I_1+I_3+I_4\|_2\right] + \frac{1}{2}\E\left[\|\nabla^2 h(I_2 + \bar{c}(I_1 + I_3+I_4))(I_1 + I_3)\|_2\|I_1 + I_3+I_4\|_2\right]\\
\nonumber & \leq M_1(h)\E\left[\|I_1+I_3+I_4\|_2\right] + \frac{M_2(h)}{2}\E\left[\|I_1 + I_3+I_4\|^2_2\right]\\
& \leq M_1(h)\E\left[\|I_1\|_2+\|I_3\|_2+\|I_4\|_2\right] + \frac{3}{2}M_2(h)\E\left[\|I_1\|^2_2 + \|I_3\|^2_2+ \|I_4\|^2_2\right].
 \end{align}
Now, note that terms involving $I_1$ and $I_3$ could be handled in  a way analogous to the proof of Theorem~\ref{thm:linearsetting}. Hence, we deal with bounding the norm of terms involving $I_4$. Furthermore, note that as a consequence of the Hessian-smoothness assumption, we also have the following to be true for any two points $\theta, \theta' \in \mathbb{R}^d$:
\begin{align*}
\| \nabla f(\theta) - \nabla f(\theta') \|_2 & \leq L \| \theta- \theta'\|_2,\\
\|\nabla f(\theta')-\nabla f(\theta)-\nabla^2 f(\theta)(\theta'-\theta)\|_2 &\le \frac{L_H}{2} \|\theta'-\theta\|_2^2.
\end{align*}
Hence, in this case, we have that
 \begin{align*}
 \|I_4\|_2  &\leq \frac{K\| \Sigma_t^{-1/2} \|_2}{t} \sum_{j=1}^{t-1} \left\| \nabla^2f(\theta^*)^{-1} + W^t_j\right\|_2 \left\| \left[\bar{R}(\Delta_j) - \nabla^2f(\theta^*) \Delta_j\right]\right\|_2\\
 & \leq  \frac{K\|\Sigma_t^{-1/2} \|_2}{t} \sum_{j=1}^{t-1}  \left\| \left[\bar{R}(\Delta_j) - \nabla^2f(\theta^*) \Delta_j\right]\right\|_2\\
 & \leq \frac{K\|\Sigma_t^{-1/2} \|_2 L_H}{2t} \sum_{j=1}^{t-1}  \left[ \| \Delta_j\|_2 \right].
 \end{align*}
 Hence, we have
 \begin{align*}
 \E \left[\|I_4\|_2 \right] \leq \frac{K\|\Sigma_t^{-1/2} \|_2L_H}{2t} \sum_{j=1}^{t-1} \E  \left[ \| \Delta_j\|_2 \right].
 \end{align*}
Now, by the proof of Lemma A.3 in~\cite{su2018statistical}, we have the following bound
\begin{align*}
\E \left[\| \Delta_j\|_2^2 \right] \leq \frac{2C}{\mu} \eta_j.
\end{align*}
Hence, we have that
\begin{align*}
\E \left[ \| \Delta_j\|_2 \right] \leq \sqrt{ \E \| \Delta_j\|_2^2} \leq \sqrt{\frac{2C}{\mu} \eta_j }.
\end{align*}
Combining the above equations, we finally have the following bound:
\begin{align*}
\E \left[\|I_4\|_2  \right] \leq \frac{\|\Sigma_t^{-1/2} \|_2}{t} \left[\frac{KL_H}{\sqrt{2\mu}}\sum_{j=1}^{t-1} \sqrt{\eta_j}\right].
\end{align*}
Similarly, we have that
 \begin{align*}
 \|I_4\|^2_2  &\leq \frac{K\|\Sigma_t^{-1/2} \|_2^2}{t^2} \sum_{j=1}^{t-1} \left\| \nabla^2f(\theta^*)^{-1} + W^t_j\right\|^2_2 \left\| \left[\bar{R}(\Delta_j) - \nabla^2f(\theta^*) \Delta_j\right]\right\|^2_2 \leq \frac{K^2 \|\Sigma_t^{-1/2} \|_2^2L^2_H}{t^2} \sum_{j=1}^{t-1}  \left[ \| \Delta_j\|^2_2 \right].
\end{align*}
To calculate the expectation, we have 
 \begin{align*}
 \E \left[\|I_4\|^2_2 \right] &\leq \frac{\|\Sigma_t^{-1/2} \|^2_2K^2 L^2_H}{t^2} \sum_{j=1}^{t-1} \E  \left[ \| \Delta_j\|^2_2 \right]  \leq\frac{\|\Sigma_t^{-1/2} \|_2}{t^2} \left[\frac{K L^2_H}{\mu} \sum_{j=1}^{t-1}   \eta_j \right].
 \end{align*}
 Going through the same steps as that in the proof of Theorem~\ref{thm:linearsetting}, we also have
 \begin{align*}
 \E\left[\|I_1\|_2^2\right] &\leq \frac{K_2^2\|\Sigma_t^{-1/2} \|^2_2\left[ \E\|\Delta_0\|_2^2\right]}{t^2\eta_0^2}\\
  \E \left[\|I_3\|_2^2\right] & \leq \frac{K_d\|\Sigma_t^{-1/2} \|^2_2~K}{t^2} \sum_{j=1}^{t-1} \left\{  \left[ e^{\left( -c_1 \sum_{i=j}^{t-1} \eta_i   \right)  }\right]^2   + \left[\frac{C'}{\eta_j^{1-c_2}} \sum_{i=j}^t m_j^t e^{-\lambda m_j^t} (m_j^i -m_j^{i-1})\right]^2 \right\}.
\end{align*}
where $K_d =\E \left[ \| X_t\|_2^2|\mathcal{F}_{t-1}\right] $. By using the above results and the fact that $\E[\|v\|_2] \leq  \left[\E \left[\|v\|_2^2 \right]\right]^{1/2}$ for a vector $v \in \mathbb{R}^d$, we get the desired result.
\end{proof}



\end{document}